\documentclass{amsart}[12pt,A4]
\usepackage[margin=2.54cm]{geometry}
\usepackage[T1]{fontenc}
\usepackage{amsthm,amsmath,amssymb,color,graphicx,url}
\usepackage{float}
\usepackage{tikz}
\usepackage{graphicx}
\usepackage{mathtools}
\usepackage{makecell}

\usepackage{pstricks-add}
\usepackage[utf8]{inputenc}
\usepackage{accents}

\usepackage{hyperref}
\hypersetup{
  hidelinks,
  colorlinks  = true,    
  urlcolor    = blue,    
  linkcolor   = blue,    
  citecolor   = blue,      
  pdfsubject  = {Maniplexes},%
  pdftitle    = {Faithfulness in Maniplexes},  %
}
\usepackage{cleveref}

\usetikzlibrary{calc}

\newtheorem{theorem}{Theorem}
\newtheorem{lemma}[theorem]{Lemma}
\newtheorem{corollary}[theorem]{Corollary}

\newtheorem{proposition}[theorem]{Proposition}
\theoremstyle{definition}
\newtheorem{definition}[theorem]{Definition}
\newtheorem{remark}[theorem]{Remark}

\newcommand{\N}{{\mathcal N}}

\newcommand{\aut}{\hbox{\rm Aut}}
\newcommand{\Aut}{\hbox{\rm Aut}}
\newcommand{\Fac}{\hbox{\rm Fac}}

\newcommand{\NN}{\mathbb{N}}
\newcommand{\ZZ}{\mathbb{Z}}

\newcommand{\stg}{\hbox{STG}}

\newcommand{\V}{{\rm V}}
\newcommand{\E}{{\rm E}}

\newcommand{\trho}{\tilde{\rho}}

\newcommand{\nset}{\{0,\ldots,n-1\}}

\newcommand{\M}{{\mathcal{M}}}
\newcommand{\F}{{\mathcal{F}}}
\newcommand{\B}{{\mathcal{B}}}
\newcommand{\C}{\mathcal{C}}

\usepackage[shadow, colorinlistoftodos,textsize=tiny, obeyFinal]{todonotes}
\setuptodonotes{fancyline, backgroundcolor=gray!10,bordercolor=gray}

\newcommand{\extm}{2^{(\M,\C)}}
\newcommand{\extmc}{(\extm)^{\omega_\C}}

\makeatletter
\numberwithin{equation}{section}
\numberwithin{figure}{section}
\numberwithin{theorem}{section}

\makeatother

\numberwithin{equation}{section}
\numberwithin{figure}{section}

\title{Highly symmetric unstable maniplexes}
\author{Isabel Hubard}
\address{Institute of Mathematics, National Autonomous University of Mexico (IM UNAM), 04510 Mexico City, Mexico}
\email{isahubard@im.unam.mx}
\author{Micael Toledo}
\address{Faculty of Mathematics and Physics, University of Ljubljana, SI-1000 Ljubljana, Slovenia and Institute of Mathematics, Physics and Mechanics, Jadranska 19, SI-1000 Ljubljana, Slovenia}
\email{micaelalexitoledo@gmail.com}

\begin{document}

\begin{abstract}
A maniplex of rank $n$ is an $n$-valent properly edge-coloured graph that generalises, simultaneously, maps on surfaces and abstract polytopes. The problem of stability in maniplexes is a natural variant of the problem of stability in graphs. A maniplex is stable if every automorphism of its canonical double cover is a lift of some automorphism of the original maniplex. Due to their very rich structure, regular (maximally symmetric) maniplexes are always stable. It is thus natural to ask what is the maximum possible degree of symmetry that a maniplex that is not stable can admit. Symmetry in maniplexes is usually measured by the number of orbits on flags (nodes) of their automorphism group. A few families of unstable maniplexes with $4$ flag-orbits are known for rank $3$. In this paper, we show that $2$-orbit maniplexes exist for every rank $n \geq 3$.
\end{abstract}

\maketitle

NOTE: this is a preliminary version, likely to be updated soon.

\section{Introduction}

Every non-bipartite graph $\Gamma$ admits a connected {\em canonical double cover} $\widetilde{\Gamma}$, isomorphic to the direct graph product $\Gamma \times K_2$. The group of automorphisms of $\widetilde{\Gamma}$ contains a subgroup isomorphic to $\Aut(\Gamma) \times \ZZ_2$, but may be, in general, significantly larger. We say $\Gamma$ is stable precisely when $\Aut(\widetilde{\Gamma}) \cong \Aut(\Gamma) \times \ZZ_2$, and say it is unstable otherwise.
Stability in graphs was first studied in \cite{Marusic1989} and have since received considerable attention (see for instance, \cite{BlasDoubleCovers,ademirdorde,QIN2019154,Surowski,Surowski2001Stability,WILSON2008359}).

The problem of stability extends naturally to other types of combinatorial objects, that can ultimately be regarded as graphs. For instance, in  \cite{gareth} Jones considers stability in maps on surfaces. Typically, a map is defined as the embedding of a connected graph on a compact surface with no boundary, with the property that the embedded graph divides the surface into simply connected regions. 
Every map uniquely determines an edge-coloured graph, called its {\em flag-graph}, that completely encodes its structure, allowing us to define maps in pure graph-theoretical terms. What is more, a map $\widetilde{M}$ is the canonical double cover of a map $\M$, in the topological sense, if and only if the flag graph of $\widetilde{M}$ is the canonical double cover of the flag graph of $\M$, in the graph theoretical sense. That is, the problem of stability in maps connects naturally to stability in graphs. There is however, a caveat. While there is a correspondence between the automorphism of a map and those of its flag graph, this correspondence is, in general, not one-to-one. The automorphisms of a map $\M$ correspond precisely to the edge-colour preserving automorphism of its flag graph, and only to those. As a consequence of this, the group of automorphism of a map, when seen as a flag-graph, acts semiregularly on its nodes, which are usually called {\em flags} in this context. This implies, as Jones remarks in \cite{gareth}, that every regular (maximally symmetric) non-orientable map is stable. Therefore, the problem of stability takes a somewhat different flavour in the setting of maps, as stable maps are very easily produced. A natural and more interesting question is whether weaker symmetry conditions also imply stability, which is partially answered in \cite{gareth}. The degree of symmetry of a map is usually measured by the number of flag-orbits of its automorphism group. Maps whose automorphism group has exactly $k$ distinct flag-orbits are called $k$-orbit maps, with $1$-orbit maps (which we call {\em regular} here) being by far the most studied class. 
Jones proves, among other things, that unstable edge-transitive maps exist. However, with the exception of a small degenerate map, having (one face, one vertex and two edges). all the maps constructed in \cite{gareth} have $4$ or more orbits on flags. The question whether $2$-orbit maps exist is left unanswered.

In this paper, we consider this question in the much broader setting of maniplexes (\cite{wilson2012maniplexes}). A maniplex of rank $n$ is a properly $n$-edge-coloured $n$-valent connected graph that generalises the notion of a flag-graph. 
In particular, a maniplex of rank $3$ is the combinatorial (graph-theoretical) equivalent of a map on a surface, in the sense that the flag-graph of a map is a $3$-maniplex, and every $3$-maniplex arises in this way. In short, maniplexes can be regarded as higher-dimension analogues of maps on surfaces (or as graph-theoretical generalisations of abstract polytopes \cite{mcmullen2002abstract}, \cite{garza2018polytopality}). 
Naturally, many of the notions (such as symmetry or orientability) of the theory of maps extend to maniplexes of higher ranks in an intuitive way. We call the nodes of a maniplex flags and we define a maniplex automorphism as an edge-colour preserving graph automorphism. 
We say a maniplex is regular (often also called reflexible) if its automorphism group is transitive on the flags. 
Just as with maps, every regular non-orientable maniplex is automatically stable and thus the problem of finding highly symmetric (but non-regular) unstable maniplexes is natural. In this paper, we prove that for rank $n\geq3$, there exist infinitely many $2$-orbit $n$-maniplexes that are unstable. Moreover, the maniplexes constructed are fully-transitive, meaning that their group of automorphisms is transitive on the set of faces of rank $i$ for every $i<n$. Precise definitions of the terms appearing in these paragraphs are provided in \Cref{sec:preliminaries} and throughout the paper. We now state the main theorem of this paper.

\begin{theorem}
\label{theo:main}
For every non-orientable regular map $\M$ of type $\{p,q\}$ where at least one of $p$ or $q$ is odd, there exists an unstable fully transitive $2$-orbit map $\M^\omega$, that covers $\M$. Moreover, for all $n>3$ there exist $2^{n-3}$ non-isomorphic unstable $2$-orbit fully transitive $n$-maniplexes, whose $3$-faces are isomorphic to $\M^\omega$.
\end{theorem}

The basis for our construction consists of two operations: the cross-cover operation for graphs, first defined in \cite{WILSON2008359}, and the colour-coded extensions for maniplexes (see \cite{twistop}), which we define here in graph-theoretical terms.

The paper is structured as follows. In \Cref{sec:preliminaries}, we give basic definitions pertaining to graphs and maniplexes. In \Cref{sec:cross} we introduce cross-covers of graph (and maniplexes) and prove a few auxiliary results relating to their connectivity and symmetry.
In \Cref{sec:ext} we discuss colour-coded extensions and define two well-known ways of extending maniplexes that will be useful for our purposes. 
In \Cref{sec:extweight} we show how the notions of cross-covers and colour-coded extensions interweave nicely, and we can use a combination of both operations to produce unstable maniplexes with desirable symmetry properties. 
In particular, we prove the two propositions that will be the key to the proof of \Cref{theo:main}, which we complete in \Cref{sec:proof}.

\section{Maniplexes}
\label{sec:preliminaries}

Maniplexes were first introduced in \cite{wilson2012maniplexes}, as generalisations of (the flag-graphs of)  maps on surfaces and abstract polytopes. 

\begin{definition}
A {\em maniplex} of rank $n$ (or simply an $n$-maniplex) is a connected $n$-valent graph $\M$ with a proper edge-colouring $\varsigma \colon \E(\M) \to \nset$ with the following {\em string property}:
   \begin{enumerate}
      \item[{\bf (S.P.)}] 
       if $i$ and $j$ are two colours such that $|i-j|>1$, then the subgraph of $\M$ induced by all the edges of colours $i$ and $j$ is a union of disjoint $4$-cycles.
   \end{enumerate}
\end{definition}
A vertex of a maniplex $\M$ is usually called a {\em flag} and an edge of colour $i$ is called an {\em $i$-edge}. If two flags are joined by an $i$-edge, we say they are $i$-adjacent or $i$-neighbours. 

For $I \subset \{0,\ldots,n-1\}$, we let $\M_I$ denote the subgraph of $\M$ induced by all the $i$-edges with $i \in I$. We let $\bar{I}$ and $\bar{i}$ denote the complements of the sets $I$ and $\{i\}$ in $\nset$, respectively. Then, $\M_{\bar{i}}$ is the subgraph of $\M$ resulting from the deletion of all the $i$-edges.
An {\em $i$-face} of $\M$, is a connected component of $\M_{\bar{i}}$. The $(n-1)$-faces of $\M$ are called the {\em facets} of $\M$ and each is a bona fide $(n-1)$-maniplex. An $i$-face and a $j$-face are said to be incident if they have non-empty intersection.

An automorphism of $\M$ is a permutation of its flags that maps $i$-adjacent flags to $i$-adjacent flags for all $i \in \{0,\ldots,n-1\}$. In other words, an automorphism of a maniplex is an edge-colour preserving graph automorphism. We let automorphisms act of the right, so $u\varphi$ denotes the image of $u$ under $\varphi$. Observe that since automorphisms preserve edge-colours, we have $u^i\varphi = (u\varphi)^i$ for all $i \in \nset$. It follows that if $\varphi$ fixes a flag $u$, then it fixes all of its neighbours and by the connectivity of $\M$ it must fix all flags. That is, the action of the group of automorphism of $\M$, denoted $\Aut(\M)$, is semiregular on the flags.

The degree of symmetry of a maniplex is usually measured by the number of flag orbits of its full automorphism group. If $\Aut(\M)$ has $k$ distinct orbits on flags, then we say $\M$ is a $k$-orbit maniplex. A $1$-orbit maniplex is maximally symmetric and is usually called regular or reflexible. 
Naturally, the most symmetric non-regular maniplexes are those having $2$ orbits on flags. However, it may me argued that some $2$-orbit maniplexes may be more symmetric than others. To better understand this, let us introduce the notions of a symmetry-type and a symmetry-type graph.

 The symmetry-type graph of $\M$, denoted by $\stg(\M)$, is the quotient of $\M$ by its full automorphism group $G = \Aut(\M)$. That is, $\stg(\M)$ is a multigraph (it may admit parallel edges and semi-edges, but no loops) and is constructed as follows. The vertices of $\stg(\M)$ are the $G$-orbits of flags of $\M$. If $u$ and $v$ are $i$-neighbours in $\M$, then we draw an edge of colour $i$ between the two vertices $u^{G}$ and $v^{G}$ in $\stg(\M)$, if they are distinct, or a semi-edge of colour $i$ at $u^{G} = v^{G}$ if they are not. The resulting multigraph is connected, $n$-valent and properly $n$-edge-coloured, and it has as many vertices as $\M$ has $\Aut(\M)$-orbits.

The symmetry-type graph of a regular maniplex consists of a single vertex with $n$ semi-edges incident to it, one for each colour in $\nset$. If $\M$ is a $2$-orbit $n$-maniplex, such as the ones we will construct in this paper, then $\stg(\M)$ consists of two vertices joined by at least one edge. If $I \subset \nset$ is the (possibly empty) set of colours of the semi-edges of $\stg(\M)$, then we say $\M$ has symmetry-type $2^n_I$, and we may use the symbol $2^n_I$ to denote its symmetry-type graph as well. Maniplexes with symmetry type $2^n_\emptyset$ are called {\emph chiral} maniplexes and have been extensively studied, especially in the context of abstract polytopes and maps.

For $i \in \nset$ we say $\M$ is $i$-face-transitive if $\aut(\M)$ acts transitively on the set of $i$-faces of $\M$. If $\M$ is $i$-face-transitive for all $i \in \nset$ then we say $\M$ is \emph{fully transitive}. Note that information about the $i$-face-transitivity of $\M$ is encoded in its symmetry-type graph. Indeed, the number of orbits of $\Aut(\M)$ on $i$-faces is equal to the number of connected components of the multigraph resulting from deleting all edges and semi-edges of colour $i$ from $\stg(\M)$. Regular maniplexes are fully-transitive, although not all fully-transitive maniplexes are regular. Figure \ref{fig:2orbstgs} shows the seven possible symmetry-type graphs for a $2$-orbit $3$-maniplex. Observe that those on the top row correspond to fully-transitive maniplexes, while those on the bottom correspond to maniplexes that are not $i$-face-transitive for some $i$. Thus, we can think of a $2$-orbit maniplex of the first kind as being 'more symmetric' than those of the second kind.

\begin{figure}[H]
    \centering
    \includegraphics[width=0.7\linewidth]{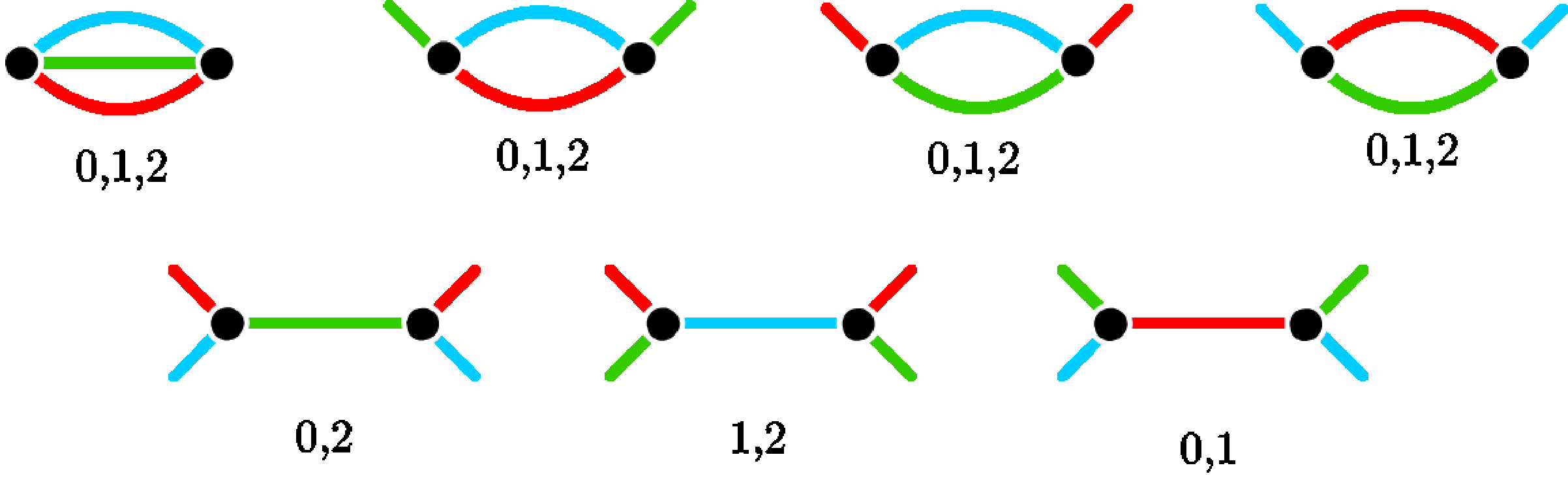}
    \caption{The seven possible symmetry-type graphs for a $2$-orbit $3$-maniplex}
    \label{fig:2orbstgs}
\end{figure}

\subsection{Stability}

A maniplex is {\em orientable} if it is bipartite, and {\em non-orientable} otherwise. This notion of orientability agrees with the notion of orientability of maps and polytopes, as a map (or abstract polytope) is orientable if and only if its flag-graph is bipartite.

The canonical double-cover of a maniplex $\M$, is the graph canonical double cover of $\M$ where every edge is given the same colour as its projection. This results in a properly $n$-edge-coloured graph $\tilde{\M}$ that satisfies the string property, but may be disconnected. In fact, the connectivity of $\widetilde{\M}$ depends entirely on the orientability of $\M$. If $\M$ is orientable, then $\widetilde{\M}$ consist of two disjoint copies of $\M$; if on the other hand $\M$ is non-orientable, then $\widetilde{\M}$ is connected and therefore a maniplex. 

We can define the canonical double cover of a non-orientable maniplex $\M$ with flag set $\F$, as the maniplex $\widetilde{\M}$ with flag-set $\F \times \ZZ_2$ and where $i$-adjacencies are given by
\begin{align*}
    (u,j)^i = (u^i,j+1).
\end{align*}

For every $\varphi \in \Aut(\M)$ and every $a 
 \in \ZZ_2$, the permutation of the flags of $\widetilde{\M}$ given by $(u,j) \mapsto (u\varphi,j+a)$ is an automorphism of $\widetilde{\M}$, called an {\em expected automorphism}. The expected automorphisms of $\widetilde{\M}$ form a subgroup of $\Aut(\widetilde{\M})$ isomorphic to $\Aut(\M) \times \ZZ_2$. The automorphisms of $\widetilde{\M}$ that are not expected are called {\em unexpected}. If all automorphisms of $\Aut(\widetilde{\M})$ are expected (and thus $\Aut(\widetilde{\M}) \cong \Aut(\M) \times \ZZ_2$), then we say $\M$ is {\em stable}; we say $\M$ is {\em unstable} otherwise. 

 Suppose $\M$ is a regular maniplex. That is, $\Aut(\M)$ acts regularly on its flag-set and thus $\Aut(\M)$ has as many elements as $\M$ has flags. Suppose $k$ is this number. Then the double cover $\widetilde{\M}$ has $2k$ flags and since the automorphism group of a maniplex is semiregular, we have $|\Aut(\widetilde{\M})|\leq2k$. However, $\widetilde{\M}$ admits $2|\Aut(\M)|=2k$ expected automorphisms, and thus $2k \leq |\Aut(\widetilde{\M})|$. We conclude $|\Aut(\widetilde{\M})| = 2k$ and thus $\M$ has to be stable.


\subsection{On maps}

If $\M$ is a finite $3$-maniplex, then there is a natural way to construct a compact closed surface with an embedded connected graph, by 'glueing' triangles together (one for each flag of $\M$; see \cite[Sect. 5]{wilson2012maniplexes}). The resulting object is a map $X$ whose flag-graph is isomorphic to $\M$. Therefore, we may think of a finite $3$-maniplex as a map.

The converse of this statement is not true in general. If $X$ is a map with a face defined by a loop, or if the embedded graph has leaves or semi-edges (which some authors allow) then the flag-graph of $X$ has at least one pair or parallel edges, and is therefore not a maniplex. 

However, there exist only two regular (very small) maps whose flag-graph is non-simple: the embedding of a graph consisting of a single vertex with a single loop into the sphere, and the embedding of the complete graph $K_2$ into the sphere. Since none of these maps is orientable, we may safely claim that every non-orientable regular map is a finite $3$-maniplex, and vice versa.

\section{Cross covers}
\label{sec:cross}

Cross-covers of graphs were first defined in \cite{WILSON2008359} and are a straightforward way to produce unstable graphs. For $k > 2$ the $k$-fold cross-cover of a graph is unstable provided it is connected and non-bipartite. When working with maniplexes, a little extra care needs to be taken, as the cross-cover of a maniplex is not always a maniplex. We will first give precise definitions of some basic notions from graph theory that will be necessary to define and discuss the cross-cover construction, in graph-theoretical term. Later on, we will see what this means in terms of maniplexes.

Let $\Gamma$ be a graph. A walk of length $k$ in $\Gamma$ is a sequence of vertices $W := (u_0,u_1,\ldots,u_k)$ such that $u_iu_{i+1}$ is an edge for all $i \in \{0,\ldots,k-1\}$. Each edge of the form $u_iu_{i+k}$ is said to be {\em traced} by $W$. A walk is closed if the first and last vertex of the sequence is the same. A cycle is a closed walk in which $u_i \neq u_j$ for any two distinct $i,j \in \{0,\ldots,k-1\}$. A path is a walk with no repeated vertices.
If the first and last vertices of a walk $W$ are $u$ and $v$, respectively, then we may say that $W$ is a $uv$-walk. Note that walks are inherently directed: they start at precise vertex, and they end (when finite) at a precise vertex. 

If $W$ and $W'$ are two walks such that the final vertex of $W$ is the initial vertex of $W'$, then we let $WW'$ denote the concatenation of $W$ with $W'$. Note that if $W$ is closed, then it can be concatenated to itself, and we let $W^k$ be the walk resulting from concatenating $W$ with itself $k$ times. The reverse of $W$, denoted $W^{-1}$, is the walk visiting the same vertices and tracing the same edges as $W$, but in reverse order.

Walks can be identified with (and are often simply defined as) the subgraph induced by the set of edges they trace. Most of the time, there is no harm in doing this, even though the sequence and the subgraph are formally distinct notions. We will adhere to this convention when convenient and when there is no possibility of ambiguity.

A graph morphism between two graphs $\Gamma$ and $\Delta$ is an adjacency preserving mapping $V(\Gamma) \to \V(\Delta)$. A morphism is an isomorphism or an epimorphism if it is bijective or surjective, respectively. We say $\Gamma$ is a cover of $\Delta$ if there is an epimorphism (which is often called a projection mapping) from $\Gamma$ to $\Delta$. An automorphism is a isomorphism of a graph onto itself. 


\begin{definition}
Let $\Gamma$ be a graph, let $k \in \NN$ and let $\omega \colon \E(\Gamma) \to \ZZ_k$. The $k$-cross-cover of $\Gamma$ relative to $\omega$, denoted $\Gamma^\omega$, is the graph such that:

\begin{enumerate}
    \item the vertex-set of $\Gamma^\omega$ is $\V(\Gamma) \times \ZZ_n$ 
    \item for every edge $e:=uv$ in $\Gamma$ and every $i \in \ZZ_n$, there is an edge in $\Gamma^\omega$ with endpoints $(u,i)$ and $(v,\omega(e)-i)$.
\end{enumerate}
\end{definition}

The function $\omega$ is called a {\em $k$-weight function} for $\Gamma$ (sometimes we way omit the prefix $k$- when not relevant). Note that the natural projection $\pi: \Gamma^\omega \to \Gamma$ mapping every vertex $(u,i)$ to $u$ is a graph epimorphism. We call the set $\pi^{-1}(u)$ the fibre of $u$. The fibre of an edge is defined similarly. If $W= (u_0,\ldots,u_m)$ is a walk of $\Gamma^\omega$, then its projection $\pi(W)$ is the sequence $(\pi(u_0),\ldots,\pi(u_m))$, which is a walk of $\Gamma$.

Even though the weight function $\omega$ is defined on the edges of $\Gamma$, it extends naturally to the set of walks of $\Gamma$. Suppose $W = (u_0,\ldots,u_{m})$ is a walk of length $k$ in $\Gamma$, and let $e_i$ denote the edge $u_iu_{i+1}$. We define the weight of $W$ as 
\begin{align}
    \omega(W) = \sum_{i=0}^{m-1} (-1)^i\omega(e_i).
\end{align}
Observe that, since walks are sequences, the weight of $W$ may differ from that of its reverse $W^{-1}$. In fact $\omega(W) = -\omega(W^{-1})$ when $W$ has even length, but $\omega(W) = \omega(W^{-1})$ when $W$ has odd length. 

We define, for $i \in \ZZ_n$, the {\em lift} of $W$ based at $(u,i)$ as the sequence 
\begin{align}
\bar{W} = ((u_0,i), (u_1,\omega(e_0)-i), (u_2,\omega(e_1) - \omega(e_0) + i), \ldots, (u_m, j)),
\end{align}
where $j = (-1)^{m}(i-\omega(W))$. Note that $W$ has one lift for each vertex in the fibre of $u_0$, and that each one is a walk of length $m$ in $\Gamma^\omega$ that projects onto $W$. Understanding how the lifts of closed walks behave will be important to prove several results regarding symmetry and connectivity. Unlike the case of regular covers (as defined in \cite{GrossVoltageGraphs}), the different lifts of a walk in a cross-cover need not be isomorphic.

\subsection{Connectivity}




\begin{lemma} 
\label{lem:evenlift}
    Let $\Gamma$ be a graph and let $\omega\colon \Gamma \to \ZZ_k$. Let $W = (u_0,\ldots,u_m)$ be a closed walk of even length $m$ in $\Gamma$ and let $a$ be the order of $\omega(W)$ in $\ZZ_k$. Then, for $i \in \ZZ_k$ and $b \in \ZZ$, the lift of $W^b$ based at $(u_0,i)$ is a closed walk of $\Gamma^\omega$ if and only if $a \mid b$.
\end{lemma}

\begin{proof}
Let $\ell_i(W^b)$ be the lift of $W^b$ based at $(u_0,i)$. Clearly, for all $j < m$, the $j$-th vertex visited by $\ell_i(W^b)$ is in the fiber of the $j$-th vertex visited by $W$. This means that $ell_i(W^b)$ visits the fiber of $u_0$ precisely at every $m$-th vertex. 

Now, the $m$-th vertex visited by $\ell_i(W^b)$ is $(u_0,j)$ where $j = (-1)^m(i-\omega(W)) = i-\omega(W)$, since $m$ is even. In general for an integer $r \leq b$, the $rm$-th vertex visited by $\ell_i(W^b)$ has second coordinate $i - r\omega(W)$. Therefore, $\ell_i(W^b)$ is closed, if and only $b\omega(W) \equiv 0 \pmod{k}$. The latter happens if and only if $a \mid b$. 
\end{proof}

\begin{remark}
\label{rem:evenlift}
In the proof of \Cref{lem:evenlift}, if $W$ is a cycle, then $\ell_i(W^b)$ is a cycle if and only if $a = b$. In this case, the number of vertices in the fiber of $u_0$ visited by $\ell_i(W^b)$ is equal to the order of $b$ in $\ZZ_n$.
\end{remark}

\begin{lemma}
\label{lem:oddlift}
Let $\Gamma$ be a graph and let $\omega\colon \Gamma \to \ZZ_n$. Let $W = (u_0,\ldots,u_k)$ be a closed walk of odd length $k$ in $\Gamma$. Let $i \in \nset$ and define
\[
a =
\begin{cases}
 1 & \text{ if $\omega(W) \equiv 2i \pmod{n}$},\\
 2 & \text{ otherwise}.
\end{cases}
\]
Then, for $b \in \ZZ$, the lift of $W^b$ based at $(u_0,i)$ is closed if and only if $a \mid b$.
\end{lemma}

\begin{proof}
For $i \in \ZZ_n$ and $r \in \ZZ$, let $\ell_i(W^r)$ denote the lift of $W^r$ based at $(u,i)$. Note that this is well defined since $W$ is closed. We have two cases.

First, suppose $\omega(W) \equiv 2i \pmod{n}$ so that $a = 1$. We will show that $\ell_i(W^b)$ is closed for all $b \in \ZZ$. Observe that $\ell_i(W)$ has final vertex $(u_0,\omega(W)-i)$, but by assumption $\omega(W)-i \equiv 2i - i = i$, and therefore $\ell(W)$ is a closed walk. Note that this implies that for every $b \in \ZZ$ we have $\ell_i(W^b) = (\ell_i(W))^b$, and $(\ell_i(W))^b$ is obviously closed, since it is the iterated concatenation of a closed walk with itself.

Now, suppose $\omega(W) \not \equiv 2i \pmod{n}$ so that $a=2$. First, note that $\ell_i(W)$ is a walk with final vertex $(u_0,\omega(W)-i)$. By assumption $\omega(W) -i \not \equiv i$, and thus $\ell_i(W)$ is not closed. However, $\ell_i(W^2)$ is simply the concatenation of $\ell_i(W)$ with $\ell_j(W)$ where $j =\omega(W)-i$. Moreover, the final vertex of $\ell_j(W)$ is $(u,r)$ where $r = \omega(W) - (\omega(W)-i) = i$. Hence, the final vertex of $\ell_i(W^2)$ is $(u_0,i)$ and thus $\ell_i(W^2)$ is a closed walk. This implies that if $b = 2s$ for some $s \in \ZZ$, then $\ell_i(W^b) = \ell_i(W^{2s}) = \ell_i(W^2)^s$ and thus $\ell_i(W^b)$ is closed.
For the converse, we will prove that if $b$ is odd, then $\ell_i(W^b)$ is not closed. Let $b = 2s + 1$ for some $s \in \ZZ$ and note that $\ell(W^{2k+1}) = \ell(W^{2k})\ell(W)$. Moreover, $\ell(W^{2k})$ is closed but $\ell(W)$ is not, and thus $\ell(W^{2k+1})$ cannot be closed. This concludes the proof.
\end{proof}

\Cref{lem:evenlift} tells us that if $W$ is a closed walk of even length $k$, then $\pi^{-1}(W)$ is a disjoint union of closed walks of length $dk$, for some divisor $d$ of $n$. Meanwhile, if $W$ has odd length $k$, then the connected components of $\pi^{-1}(W)$ are not necessarily all isomorphic. By \Cref{lem:oddlift}, some are closed walks of length $k$ and some others are closed walks of length $2k$, and these are the only possibilities for a connected component of $\pi^{-1}(W)$.
Lemma \ref{lem:evenlift} has the following important consequence.

\begin{lemma}
\label{lem:connected}
Let $\Gamma$ be a connected graph and let $\omega\colon \Gamma \to \ZZ_k$. Let $C$ be a cycle of even length $m$ in $\Gamma$. If $\gcd(\omega(C),k) = 1$, then $\Gamma^\omega$ is a connected graph.
\end{lemma}

\begin{proof}
Let $C=(u_0,\ldots,u_m)$. By \Cref{lem:evenlift} and \Cref{rem:evenlift}, the lift of $C^k$ based at $(u_0,0)$ is a cycle, and it visits all the vertices in the fibre of $u_0$. Therefore, it suffices that we show that if $(v,i)$ is a vertex of $\Gamma^\omega$ there is a walk connecting it to some vertex in the fibre of $(u_0,0)$. Let $(v,i)$ is a vertex of $\Gamma^\omega$. Since $\Gamma$ is connected, there is a $vu$-path $W$ and the lift of $W$ based at $(v,i)$ ends in a vertex in the fiber of $u_0$.
\end{proof}

\subsection{Automorphisms}

\begin{lemma}
\label{lem:isolift}
Let $\M$ and $\M'$ be graphs with weight functions $\omega: \E(\M) \to \ZZ_k$  and $\omega': \E(\M') \to \ZZ_k$. If there exists a graph isomorphism $\varphi \colon \M \to \M'$ and a group isomorphism $\alpha \in \Aut(\ZZ_k)$ such that 
\begin{align*}
\omega'(e\varphi) = \alpha(\omega(e)).
\end{align*} 
Then, the function $\bar{\varphi}$ given by  $(u,i) \mapsto (\varphi(u),\alpha(i))$ is an isomorphism between $\M^\omega$ and $(\M')^{\omega'}$.
\end{lemma}

\begin{proof}
Let $\bar{e}$ be an edge of $\M^\omega$. Then $\bar{e}$ must be in the fiber of some edge $e$ of $\M$ with endpoints, say, $u$  and $v$. That is, the endpoints of $\bar{e}$ are $(u,i)$ and $(v,\omega(e)-i)$ for some $i \in \ZZ_n$. 

Now, since $(u\varphi,v\varphi)$ is an edge of $\M$ with weight $\omega(e\varphi)$, we see that $(u\varphi,\alpha(i))$ (which equals $(u,i)\bar{\varphi}$) is adjacent to $(v\varphi,\omega(e\varphi)-\alpha(i))$. However,
\begin{align*}
(v\varphi,\omega(e\varphi)-\alpha(i)) &= (v\varphi,\alpha(\omega(e))-\alpha(i))\\
 &= (v\varphi,\alpha(\omega(e)-i))\\
 &= (v,\omega(e)-i)\bar{\varphi}
\end{align*}
where the first equality follows from the fact that $\omega(e\varphi) = \alpha(\omega(e))$, and the second one from the fact $\alpha$ is a group automorphism. This shows that $\bar{\varphi}$ maps edges to edges. Furhermore, since it is clearly bijective and the graphs $\M^\omega$ and $\M'^{\omega'}$ have the same number of edges, $\bar{\varphi}$ is a graph isomorphism.
\end{proof}

We say an automorphism $\varphi \in \Aut(\Gamma)$ {\em lifts} if there is an automorphism $\bar{\varphi} \in \Aut(\Gamma^\omega)$ such that $\pi(u\bar{\varphi}) = \pi(u)\varphi$. In that case, we call $\bar{\varphi}$ a {\em lift} of $\varphi$ and we say $\omega$ is $\varphi$-consistent. If all automorphisms of $\Gamma$ lift, then we say $\omega$ is $\Aut(\Gamma)$-consistent. Lemma \ref{lem:isolift} gives us a simple criterion for automorphism of a graph to lift. 

\begin{corollary}
\label{cor:autolift}
Let $\M$ be a graph with a weight function $\omega: \E(\M) \to \ZZ_k$. If there exists  $\varphi \in \Aut(\M)$ and $\alpha \in \Aut(\ZZ_k)$ such that $\omega(e\varphi) = \alpha(\omega(e))$ for all edges $e$ of $\M$, then $\varphi$ has a lift.
\end{corollary}

\begin{proof}
By \Cref{lem:isolift} the mapping $(u,i) \mapsto (u\varphi,\alpha(i))$ is an automorphism of $\M^\omega$ and it clearly commutes with the projection $\pi$.
\end{proof}

\subsection{Cross-covers of maniplexes}

Let $\M$ be an $n$-maniplex and let $\omega$ be a weight function for $\M$. Just as with the canonical double covers of maniplexes, the colouring of $\M$ extends naturally to its cross-cover $\M^\omega$ by simply colouring each edge $e$ of $\M^\omega$ with the colour of $\pi(e)$, making $\M^\omega$ a properly $n$-edge-coloured graph, but not necessarily a maniplex. Indeed, may be disconnected or may lack the string property. Nevertheless, we will henceforth assume, whenever we are dealing with the cross-cover of a maniplex, that its edges are properly coloured in this way. After all, Lemma \ref{lem:connected} gives us a criterion for connectivity, and we can derive sufficient and necessary conditions for the cross-cover of a maniplex to have the string property from Lemma \ref{lem:evenlift}, as is shown in \Cref{lem:cuadrados}.

\begin{lemma}
\label{lem:cuadrados} 
Let $\M$ be an $n$-maniplex with a weight function $\omega$. Then $\M^\omega$ has the string property if and only if $\omega(C)=0$ for every $4$-cycle $C$ in $\M$ of alternating non-consecutive colours.
\end{lemma}

Let $W=(u_0,\ldots,u_m)$ be a walk of length $k$ in $\M$ and let $c_i$ be the colour of the edge $u_{i-1}u_i$. In other words, $W$ traces edges of colours $c_1,c_2,\ldots,c_m$, in that order. Then, every lift of $W$ must be a walk that traces edges following the same colour sequence. This means that if $\varphi \in \Aut(\M^\omega)$ maps a vertex $(u_0,s)$ to $(u_0,t)$ (both in the fiber of $u_0$), then it must map the lift of $W$ based at $(u_0,s)$ to the lfit of $W$ based at $(u_0,t)$. This rather simple observation gives us the following useful lemma. 

\begin{lemma}
    \label{lem:noregular}
Let $\M$ be an $n$-maniplex and let $\omega\colon \M \to \ZZ_k$ with $k \geq 3$, such that $\M^\omega$ is a maniplex. Suppose there exists a closed walk of odd length $W = (u_0,\ldots,u_r)$ in $\M$. If $\omega(W)$ is even, then $\M^\omega$ is not regular.
\end{lemma}

\begin{proof}
Suppose $\omega(W) \equiv 2r$ for some integer $r \leq \lfloor \frac{k}{2} \rfloor$. Then, by Lemma \ref{lem:oddlift}, the lift of $W$ at $(u_0,r)$ is a closed walk, since $\omega(W) \equiv 2r$. Set $s = r+1$ and observe that $\omega(W) \not\equiv 2s$ and thus the lift of $W$ at $(u_0,s)$ is an open walk. Hence, no automorphism of $\M^\omega$ can map  $(u_0,r)$ to $(u_0,s)$, as it would map a closed walk to an open walk. Therefore, $\M^\omega$ is not regular.
\end{proof}

The following theorem is a particular case of \cite[Theorem 3]{WILSON2008359}, which states that if $\omega$ is an $k$-weight function, $k>2$, for a graph $\Gamma$, then the cross-cover $\Gamma^\omega$ is unstable. The theorem is proved in \cite{WILSON2008359} by exhibiting an unexpected automorphism in the canonical double cover of $\Gamma^\omega$. In the particular case where $\Gamma^\omega$ is a maniplex, this automorphisms preserves the colouring of the edges, and thus, it is also a (maniplex) automorphism of the canonical double cover. 


\begin{theorem}
\label{theo:cross}
Let $\M$ be an $n$-maniplex and let $\omega \colon \M \to \ZZ_k$ with $k \geq 3$. If $\M^\omega$ is a non-orientable maniplex, then $\M^\omega$ is unstable.
\end{theorem}

\section{Extensions}
\label{sec:ext}

An extension of an $n$-maniplex $\M$ is an $(n+1)$-maniplex whose facets are all isomorphic to $\M$. Although several constructions (see for instance \cite{cayleyext,twistop}) exist for producing extensions of maniplexes, we will use the colour-coded extension method defined in \cite{twistop}. We will give a definition of the construction and provide some basic facts related to it that will be useful to us. Further details can be found in \cite{twistop}.

Let $\M$ be an $n$-maniplex and let $\F_\M$ and $\textrm{Fac}(\M)$ denote the set of flags and the set of facets of $\M$, respectively. For a positive integer $\ell$, an {\em $\ell$-colouring} for $\M$ is a function $\mathcal{C} \colon \textrm{Fac}(\M) \to \{1,\ldots,\ell\}$. The {\em colour} of a facet is its image under $\C$ and the colour of a flag is simply the colour of the facet that contains it. 

The colour-coded extension of $\M$ relative to an $\ell$-coloruing $\mathcal{C}$, or simply the $\C$-extension of $\M$, is the $(n+1)$-maniplex $2^{(\M,\C)}$ with flag-set $\F_\M \times \ZZ_2^\ell$ where the $i$-neighbours of a flag $(u,x)$ are given by

\[
(u,x)^i=
\begin{cases}
(u^i,x) & \text{if $i<n$},  \\
(u,x^j) & \text{if $i = n$},
\end{cases}\\
\]
where $j$ is the colour of $u$, and $x^j$ is the element of $\ZZ_2^\ell$ that differs from $x$ only in its $j$-th entry.

The $\C$-extension $2^{(\M,\C)}$ is a well-defined $(n+1)$-maniplex. It has $2^\ell$ facets, all isomorphic to $\M$. To see this, note that if $x \in \ZZ_2^\ell$, then the subgraph $F_x$ induced by all flags whose second coordinate is $x$, only has $i$-edges with $i < n$. In fact, $F_x$  is a connected component of the subgraph of $2^{(\M,\C)}$ induced by all $i$-edges with $i < n$. In other words, each $F_x$ is a facet of $\extm$, and is therefore an $n$-maniplex in its own right. One can verify that the projection $\tilde{\pi} \colon \extm \to \M$ (restricted to $F_x$) maps $F_x$ to $\M$ isomorphically.

A symmetry $\varphi$ of $\M$ extends naturally to a symmetry of $2^{(\M,\C)}$ when it is `compatible' with the colouring $\C$. Let us be more precise by introducing some terminology. We say $\C$ is {\em $\varphi$-invariant} if $\varphi$ induces a permutation of the chromatic classes of $\M$. That is, if the mapping $\C(F) \mapsto \C(\varphi(F))$ is a well-defined bijection of the set of colours $\{1,\ldots,\ell\}$ onto itself. We may slightly abuse notation, and let $\varphi(c)$ denote the image of a colour $c$ under the permutation induced by $\varphi$. Further, if $x \in \ZZ_2^\ell$ and $x_i$ denotes the $i$-th coordinate of $x$, we let $\varphi(x)$ be the element $(x_{\varphi^{-1}(1)},x_{\varphi^{-1}(2)},\ldots,x_{\varphi^{-1}(\ell)})$. 


With the notation from the above paragraph, if $\C$ is $\varphi$-invariant, then $\varphi$ extends to an automorphism $\tilde{\varphi}$ of $2^{(\M,\C)}$  given by
\begin{align}
\label{eq:tildes}
(u,x)\tilde{\varphi}= (u\varphi, \varphi(x)).
\end{align}

If $\C$ is $\varphi$-invariant for all $\varphi \in \Aut(\M)$, then we say $\C$ is {\em $\Aut(\M)$-invariant}. In that case, all automorphisms of $\M$ extend to an automorphism of $\extm$. However, it is important to point out that $\extm$ also admits automorphisms that do not come from automorphisms of $\M$. For instance, for all $j \in \{1,\ldots,\ell\}$ the mapping 
\begin{align}
\label{eq:taus}
\tau_j \colon (u,x) \mapsto (u,x^j)    
\end{align}
is an automorphism of $\extm$ that maps every flag of colour $j$ to its $n$-adjacent flag.

This means that highly symmetric maniplexes extend to highly symmetric maniplexes when the colouring is $\Aut(\M)$-invariant. 



As we will see in \Cref{sec:extweight}, $\Aut$-invariant colourings will play a central role in proving \Cref{theo:main}. Naturally, the number of $\Aut(\M)$-invariant colourings that can be defined for a given maniplex $\M$ depends on the action of $\Aut(\M)$ on the set of facets of $\M$. However, all maniplexes admit one $\Aut$-invariant colouring (yielding non-degenerate maniplexes), regardless of this action.\\

{\bf [C.1] Total colouring.}
Let $\M$ be a maniplex and consider an $\ell$-colouring $\hat{\C}$ where $\ell$ is the number of facets of $\M$. That is, all facets of $\M$ receive different colours. The $\hat{\C}$-extension of $\M$ is usually denoted $\hat{2}^{\M}$. This extension was first defined in \cite{DANZER1984115} for polytopes, and was generalised to maniplexes in \cite{twistop}, and has since been used frequently to construct regular maniplexes.

In the particular case where $\M$ is itself a colour-coded extension of some $(n-1)$-maniplex $\mathcal{N}$, then we have some information about the action of $\Aut(\M)$ on its facets, and we can guarantee that $\M$ admits an additional $\Aut(\M)$-invariant colouring, which we describe below.\\

{\bf [C.2] Antipodal colouring.}
For $x \in \ZZ_2^\ell$, we define the {\em antipode} of $x$ as the unique element of $\ZZ_2^\ell$ that differs from $x$ in every entry. Clearly, antipodality is an equivalence relation with classes of equivalence of size $2$. We say two facets of $\extm$, $F_x$ and $F_y$, are {\em antipodal} if $x$ and $y$ are antipodal elements of $\ZZ_2^\ell$. An {\em antipodal colouring} of $\extm$ is a colouring in which two facets receive the same colour if and only if they are antipodal. All antipodal colourings are equivalent, and thus it makes sense to talk about {\em the} antipodal colouring of $\extm$, denoted $\C^{\times}$.\\

\begin{lemma}
Let $\N$ be an $(n-1)$-maniplex with an $\ell$-colouring $\C$ and let $\M := 2^{(\N,\C)}$. Then the antipodal colouring $C^\times$ of $\M$ is $\Aut(\M)$-invariant.    
\end{lemma}

\begin{proof}
First, observe that the group  $H = \langle \tau_1,\ldots,\tau_\ell \rangle$, where each $\tau_i$ is the symmetry defined in \Cref{eq:taus}, acts transitively on the facets of $\M$. Clearly, the full group $\Aut(\M)$ is facet-trasitive as well. It follows that a colouring of (the facets of ) $\M$ is $\Aut(\M)$-invariant if and only if the chromatic classes form a system of blocks of imprimitivity. 

Now, consider the graph $Q_\ell$ whose vertices are the facets of $\M$, where two vertices are adjacent if the corresponding facets are joined by an $n$-edge in $\M$ (that is, if there in an $n$-edge having one endpoint in the first facet, and the other in the second facet). The graph $Q_\ell$ is isomorphic to the $\ell$-cube graph and every automorphism of $\M$ (which permutes its facets) induces a permutation of the vertices of $Q_\ell$. Moreover, since an automorphism of $\M$ maps $n$-edges to $n$-edges, the induced permutation on the vertices of $Q_\ell$ must map adjacent vertices to adjacent vertices. In other words, every automorphism of $\M$ induces a graph automorphism of $Q_\ell$. Let $G = \Aut(\M)$ and let $G'\leq\Aut(Q_\ell)$ be the group of automorphisms of $Q_\ell$ induced by $G$. Now consider the group $H = \langle \tau_1,\ldots,\tau_\ell \rangle$ and observe that $H$ acts not only transitively, but regularly, on the set of facets of $\M$. It follows that the group $H'$ of automorphisms of $Q_\ell$ induced by $H$ acts regularly on the vertices of $Q_\ell$. Moreover, one can see that the set of all pairs of antipodal vertices of $Q_\ell$ is a block system for $H'$ (indeed, $H'$ acts on the main diagonals of the $\ell$-cube). Since this is also a block system for the full automorphism group of $Q_\ell$, and $H' \leq G' \leq \Aut(Q_\ell)$, we see that it is a block system for $G'$ too. This block system corresponds to the partition into chromatic classes (of the facets of $\M$) with the antipodal colouring $\C^\times$ of $\M$, and thus, this partition must be a block system for $\Aut(\M)$ (acting on facets). We conclude that the antipodal colouring is $\Aut(\M)$-invariant.
\end{proof}

\begin{remark}
If $\N$ is an $n$-maniplex and $\M$ is an extension of $\N$ with $k>1$ facets, then the extensions of $\M$ obtained by the total and antipodal colourings are non-isomorphic. This is easily seen by noting that the number of facets of $\hat{2}^\M$ is $2^k$,  while $2^{(\M,\C^\times)}$ has $2^{k/2}$ facets.
\end{remark}

\section{Extended weight functions}
\label{sec:extweight}

We know how to extend maniplexes to obtain maniplexes of higher rank by means of colourings of their facets. In this section we will describe a way to extend weight functions as well. We will want, whenever we are extending a maniplex $\M$ with a colouring $\C$, to extend any weight function of $\M$ to a weight function of $\extm$ in a way that some properties of the cross-cover of $\M$ also hold for the cross-cover of $\extm$. In particular, if the cross-cover of $\M$ is a non-orientable maniplex, then so will be the cross-cover of $\extm$.

We define the {\em parity} of an element $x \in \ZZ_2^k$, denoted $\sigma(x)$, as $(-1)^b$ where $b$ is the
number of entries of $x$ that are equal to $1$. In other words, $\sigma(x)$ is $1$ if $x$ has an even number of entries equal to $1$, and is $-1$ otherwise. The parity of a flag $(u,x)$ of $2^{(\M,\C)}$ is simply the parity of $x$. Observe that the endpoints of an $i$-edge have distinct parity if $i=n$, but have the same parity if $i < n$. We can thus define the parity of an $i$-edge whenever $i < n$, as the parity of its endpoints.

Let $\M$ be an $n$-maniplex with an $\ell$-colouring $\C$ and a weight function $\omega\colon \M \to \ZZ_k$. Let $\tilde{\pi}$ denote the projection $2^{(\M,\C)} \to \M$. We define the extension of $\omega$ relative to $\C$ as the function 
$\omega_\C \colon \extm \to \ZZ_k$ mapping each $i$-edge $e$ of $2^{(\M,\C)}$ to 

\[
\omega_\C(e)=
\begin{cases}
\sigma(e)\omega(\tilde{\pi}(e)) & \text{ if $i < n$},   \\
0 & \text{ if $i = n$}.
\end{cases}\\
\]

The weight function $\omega_\C$ extends $\omega$ in the following sense. If $F_x$ is the facet of $\extm$ associated to an even $x \in \ZZ_2^\ell$, then not only is $F_x$ a maniplex isomorphic to $\M$, but the functions $\omega_\C$ (restricted to $F_x$) and $\omega$ coincide.

Consider a walk $\tilde{W}$ in $\extm$ that traces no edges of colour $n$. Then, $\tilde{W}$ is contained within a facet $F_x$ of $\extm$ and thus, all edges it traces have the same parity $\sigma(x)$. By definition all even edges have the same weight as their projection, while the weight of an odd edge is the inverse of that of its projection. Therefore, if $W = \tilde{\pi}(\tilde{W})$ we have
\begin{align}
\label{eq:pesocaminosext}
\omega_\C(\tilde{W}) = \sigma(x)\omega(W).
\end{align}

One of the implications of this observation is that maniplexes whose cross-cover has the string property extend to maniplexes with this same attribute.

\begin{lemma}
\label{lem:extstring}
Let $\M$ be an $n$-maniplex, let $\C \colon \textrm{Fac}(\M) \to \{1,\ldots,\ell\}$ be a colouring of the facets of $\M$ and let $\omega \colon \E(\M) \to \ZZ_k$. If $\M^\omega$ has the string property, then $(2^{(\M,\C)})^{\omega_\C}$ has the string property. 
\end{lemma}

\begin{proof}
Let $i,j \in \{0,\ldots,n\}$ be non-consecutive colours and consider a $4$-cycle $C$ in $2^{(\M,\C)}$ of alternating colours $i$ and $j$. By Lemma \ref{lem:cuadrados}, it suffices to show that $\omega_\C(C) \equiv 0 \pmod k$. We have two cases. 



First suppose that $i,j < n$. Then $C$ must be of the form $((u,x),(u^i,x),(u^{ij},x),(u^{j},x),(u,x))$, for some flag $u$ of $\M$ and some $x \in \ZZ_2^\ell$. Moreover $C' = (u,u^i,u^{ij},u^{j},u)$ is a $4$-cycle of $\M$ of alternating colour $i$ and $j$ and by Equation \ref{eq:pesocaminosext} we have $\omega_\C(C) \in \{\omega(C'),-\omega(C')\}$. However, since $\M^\omega$ has the string property, then by Lemma \ref{lem:cuadrados}, $\omega(C')\equiv 0 \pmod k$ and thus $\omega_\C(C) \equiv 0$.

Now, suppose that $i < n$ and $j = n$. Then $C=((u,x),(u^i,x),(u^{in},y),(u^{n},y),(u,x))$ for some $u \in \F_\M$ and $x,y \in \ZZ_2^\ell$. Note that  $x$ and $y$ must have different parity and thus the $i$-edge $e_x$ joining $(u,x)$ to $(u^i,x)$ and the $i$-edge $e_y$ joining $(u^{in},y)$ to $(u^{n},y)$ have weights that are inverse to each other. Since the two remaining edges of $C$ have colour $n$, they must have weight $0$, by the definition of $\omega_\C$. Therefore $\omega_\C(C)= \omega_\C(e_x) + \omega_\C(e_y) = 0$.
\end{proof}

Recall that the facets of $\extm$ are in correspondence with the elements of $\ZZ_2^\ell$, and that we denote by $F_x$ the facet containing all flags with second coordinate $x$. The following lemma will help us better describe (the facets of) the cross-cover $(\extm)^{\omega_\C}$. We may allow a slight abuse of notation and use the symbol $\omega$ to denote both the weight function on $\M$ and its restriction to a subgraph of $\M$. Then, if $\mathcal{N} \leq \M$ the graph  $\mathcal{N}^\omega$ is a well-defined cross-cover.

\begin{lemma}
 \label{lem:facetasbonitas}
Let $\M$ be an $n$-maniplex with a weight function $\omega \colon \E(\M) \to \ZZ_k$ and a colouring $\C\colon \M \to \{1,\ldots,\ell\}$. Let $x \in \ZZ_2^\ell$ and let $F_x$ be the corresponding facet of $\extm$.  Then $(F_x)^{\omega_\C} \cong \M^\omega$. 
\end{lemma}


\begin{proof}
Observe that $F_x$ is isomorphic to $\M$ via the projection $\tilde{\pi} \colon (u,x) \mapsto u$. Let $\varphi$ denote this isomorphism, so that $\tilde{\pi}(\tilde{e}) = \tilde{e}\varphi$ for all edge $\tilde{e}$ of $F_x$. Further, let $\alpha\colon \ZZ_k \to \ZZ_k$ be given by $\alpha(i) = \sigma(x)i$. Clearly, $\alpha \in \Aut(\ZZ_k)$ and if $\tilde{e}$ is an edge of $F_x$, then 
\begin{align*}
\alpha(\omega_\C(\tilde{e})) = \sigma(x)\omega_\C(\tilde{e}) = \sigma(x)\sigma(x)\omega(\tilde{\pi}(\tilde{e})) =  \omega(\tilde{\pi}(\tilde{e})) = \omega(\tilde{e}\varphi).  
\end{align*}
Then, by Lemma \ref{lem:isolift}, $\M^\omega \cong (F_x)^{\omega_\C}$.
\end{proof}

\begin{corollary}
\label{cor:extno}
If $\M^\omega$ is non-orientable, then so is $(\extm)^{\omega_\C}$.
\end{corollary}

\begin{lemma}
\label{lem:extcon}
Let $\M$ be an $n$-maniplex, let $\C$ be an $\ell$-colouring of the facets of $\M$ and let $\omega$ be a $k$-weight function for $\M$. If $\M^\omega$ is connected, then $(2^{(\M,\C)})^{\omega_\C}$ is also connected.
\end{lemma}

\begin{proof}

Let $x \in \ZZ_2^\ell$ and let $F_x$ be the corresponding facet of $\extm$. By Lemma \ref{lem:facetasbonitas}, $(F_x)^{\omega_C}$ is isomorphic to $\M^\omega$. In particular, $(F_x)^{\omega_C}$ is connected. Therefore, to show that $(\extm)^{\omega_\C}$ is connected it suffices that we show that for all $y \in \ZZ_2^\ell$, we can connected $(F_x)^{\omega_C}$ to $(F_y)^{\omega_C}$ (which is also connected subgraph of $(2^{(\M,\C)})^{\omega_\C}$). We will show that for all $j \in \{1,\ldots,\ell\}$ there is an $n$-edge between a flag in $(F_x)^{\omega_C}$ and a flag in $(F_{x^j})^{\omega_C}$, where $x^j$ is the element of $\ZZ_2^\ell$ differing from $x$ only in its $j$-th entry.

Now, let us get our attention to $\M$ and the colouring function $\C$. By definition $\C$ is surjective, there must be a flag $u \in \M$ such that $\C(u) = j$. This means that $\tilde{u} := (u,x)$ is a flag in $F_x$ and its $n$-neighbour is $\tilde{u}^n = (u,x^j)$, which is a flag in $F_{x^j}$. Finally, $(\tilde{u},0)$ is a flag of $(F_x)^{\omega_C}$ and its $n$-neighbour is $(\tilde{u}^n,0)$, since all $n$-edges have weight $0$ with the function $\omega_\C$. Since $(\tilde{u}^n,0)$ is a flag of $(F_{x^j})^{\omega_C}$, we conclude that $(\extm)^{\omega_\C}$ is connected.
\end{proof}

Lemmas \ref{lem:extstring} and \ref{lem:extcon}, together with Corollary \ref{cor:extno}, show that the extended colouring $\omega_\C$ behaves precisely like we need.

\begin{proposition}
\label{prop:pares}
Let $\M$ be an $n$-maniplex, let $\C$ be an $\ell$-colouring of the facets of $\M$ and let $\omega$ be a $k$-weight function for $\M$. If $\M^\omega$ is a non-orientable maniplex, then $(2^{(\M,\C)})^{\omega_\C}$ is a non-orientable maniplex. 
\end{proposition}

\Cref{prop:pares} tells us that it suffices to find a map $\M$ with a weight function $\omega$ such that $\M^\omega$ is a non-orientable maniplex (which must necessarily be unstable by \Cref{theo:cross}), to show that there are unstable maniplexes of every rank $n \geq 3$. Finding such a pair $(\M,\omega)$ is in fact not difficult, as we will see in the following pages. However, the maniplexes obtained by extending $\M$ and $\omega$ may,  in principle, be of any symmetry type. In fact, they may very well completely lack symmetry. Since we are interested in maniplexes that are both unstable and very symmetric, extra conditions need to be imposed when considering a pair $(\M,\omega)$. This motivates the following definition.



\begin{definition}
\label{def:properpair}
Let $\M$ be an $n$-maniplex and let $\omega\colon \E(\M) \to \ZZ_k$ be a weight function for $\M$. We say $(\M,\omega)$ is a {\em proper pair} if $\M$ and $\omega$ satisfy the following:
\begin{enumerate}
    \item $\M$ is regular,
    \item $\M$ admits a closed walk $W$ of odd length such that $\omega(W)$ is even,
    \item $\M^\omega$ is a non-orientable maniplex,
    \item $\omega$ is $\Aut(\M)$-consistent.
\end{enumerate}
\end{definition}

As we will see in the following two lemmas, not only do proper pairs produce unstable $2$-orbit maniplexes, but this property is hereditary, in the sense that proper pairs extend to proper pairs when the colouring function defining the extension is $\Aut$-invariant.

\begin{lemma}
\label{lem:niceext}
Let $\M$ be an $n$-maniplex and let $\omega\colon \M \to \ZZ_k$ be a weight function for $\M$. Suppose $(\M,\omega)$ is a proper pair. If $\C$ is an $Aut(\M)$-invariant $\ell$-colouring of $\M$, then $(\extm,\omega_\C)$ is a proper pair.
\end{lemma}

\begin{proof}
We need to show that $(\extm,\omega_\C)$ satisfies items (1)--(4) of \Cref{def:properpair}. We get items (2) and (3) almost for free. Indeed, $(\extm) ^{\omega_C}$ is a non-orientable maniplex because of \Cref{prop:pares} and thus (3) holds. Furthermore, if $C=(v_0,\ldots,v_r)$ is an odd cycle in $\M$ with $\omega(C)$ even, then $\tilde{C} = ((v_0,\bar{0}),\ldots,(v_r,\bar{0}))$ is a cycle of odd length in $\extm$ and $\omega_\C(\tilde{C}) = \omega(C)$. This gives us property (2).

Now, to show that $\extm$ is regular (item (1)), we will give a set of standard generators $\trho_0,\ldots,\trho_n$ for $\Aut(\extm)$. Let $u_0$ be the base flag of $\M$ and let $\rho_0, \ldots, \rho_{n-1}$ be the standard generators of $\Aut(\M)$ relative to $u_0$.  For $i \in \{0,\ldots,n-1\}$, let $\trho_i$ be defined as in \Cref{eq:tildes}. That is
\begin{align}
    (u,x)\trho_i = (u\rho_i,\rho_i(x)).
\end{align}
Observe that $\trho_i$ is an automorphism of $\extm$ mapping the flag $(u_0,\bar{0})$ to its $i$-neighbour $(u_0^i,\bar{0})$. 

To define the last generator, let $s=\C(u_0) $ and define 
\begin{align*}
\trho_n = \tau_s,
\end{align*}
where $\tau_s$ is as in \Cref{eq:taus}. Note that $\trho_n$ maps $(u_0,\bar{0})$ to its $n$-neighbour. This shows that $\{\trho_0,\ldots,\trho_n\}$ is the set of standard generator of $\extm$ relative $(u_0,\bar{0})$, and thus $\extm$ is regular.

Lastly, we will show that $\omega_\C$ is $\Aut(\extm)$-consistent. For this, it suffices to show that for all $i \in \{0,\ldots,n\}$, $\omega_\C$ is $\trho_i$-consistent. That is, that there exists $\tilde{\alpha}_i \in \Aut(\ZZ_k)$ such that
\begin{align}
\label{eq:alfais}
    \omega_\C(\tilde{e}\trho_i) = \tilde{\alpha}_i(\omega_\C(\tilde{e}))
\end{align}
for all edges $\tilde{e}$ of $\extm$. Note that if $\tilde{e}$ is an $n$-edge, then \Cref{eq:alfais} holds for all $\alpha_i \in \Aut(\ZZ_k)$, since all $n$-edge have weight $0$ and each $\trho_i$ maps $n$-edges to $n$-edges. Therefore, we may only concern ourselves with $j$-edges where $j<n$. 

Since $\trho_n$ is defined differently than the rest of the generators $\trho_i$, we will deal with it separately. 

For now, let $i \in \{0,\ldots,n-1\}$. Since $\omega$ is $\Aut(\M)$-consistent, there exists $\alpha_i \in \Aut(\M)$ such that $\omega(e\rho_i) = \alpha_i(\omega(e))$ for all edges $e$ of $\M$. Define $\tilde{\alpha}_i = \alpha_i$.

Let $j \in \{0,\ldots,n-1\}$ and let $\tilde{e}$ be a $j$-edge of $\extm$. Then its endpoints are $(u,x)$ and $(u^j,x)$ for some $u \in \F_\M$ and some $x \in \ZZ_2^\ell$. Without loss of generality, assume that $\sigma(x) = 1$. Let $e$ be the $j$-edge $uu^j$ of $\M$. Then, 
\begin{align}
\label{eq:pares1}
\omega_\C(\tilde{e}) = \omega(e).
\end{align}
Moreover, $\tilde{\pi}(\tilde{e}\trho_i) = e\rho_i$ and by the definition of $\omega_\C$, we have 
\begin{align*}
\omega_\C(\tilde{e}\trho_i) = \sigma(\tilde{e})\omega(e\rho_i) = \sigma(\rho_i(x))\omega(e\rho_i).
\end{align*}
However $\sigma(\rho_i(x)) = \sigma(x) = 1$ since $\rho_i$ acts on $x$ by permuting its coordinates. Thus, 
\begin{align}\label{eq:pares2}
\omega_\C(\tilde{e}\trho_i) = \omega(e\rho_i).
\end{align}
Then, we have
\begin{align*}
    \omega_\C(\tilde{e}\trho_i) = \omega (e\rho_i) = \alpha_i(\omega(e)) = \alpha_i(\omega_\C(\tilde{e})) = \tilde{\alpha}_i(\omega_\C(\tilde{e})),
\end{align*}
where the first and third equality follow from \Cref{eq:pares1} and \Cref{eq:pares2}, respectively, and the second and fourth follow from the definition of $\tilde{\alpha_i}$.

It only remains to see what happens with $\trho_n$. With the notation above, $\trho_n$ maps $(u,x)$ to $(u,x^s)$ and $(u^j,x)$ to $(u^j,x^s)$. Then, $\trho_n$ inverts the sign of $\tilde{e}$. In other words, $\omega_\C(\tilde{e}) = \omega(e)$ but $\omega_\C(\tilde{e}\trho_n) = -\omega(e)$. Define $\tilde{\alpha}_n$ as the automorphism of $\ZZ_k$ that maps every $a$ to $-a$. We have

\begin{align}
    \omega_C(\tilde{e}\trho_n) = -\omega (e) = \tilde{\alpha}_n(\omega(e)) = \tilde{\alpha}_n(\omega_C(\tilde{e})).
\end{align}
We conclude that $\omega_\C$ is $\Aut(\extm)$-consistent. The result follows.
\end{proof}

\begin{lemma}
\label{lem:nicecover}
Let $\M$ be a regular $n$-maniplex and let $\omega\colon \M \to \ZZ_k$ be a weight function for $\M$. If $(\M,\omega)$ is a proper pair and $k=4$, 
then $\M^\omega$ is an unstable $2$-orbit maniplex. 
\end{lemma}

\begin{proof}
Let $u_0$ be a fixed flag of $\M$ and let $\rho_0,\ldots,\rho_{n-1}$ be a set of standard generators for $\Aut(\M)$. Since $\omega$ is $\Aut(\M)$-consistent, then each $\rho_i$ lifts to an automorphism $\bar{\rho_i}$ of $\M^{\omega}$. Note that $\bar{G} := \langle \bar{\rho_0},\ldots,\bar{\rho_{n-1}} \rangle$ acts on the set of fibers of $\M^\omega$, and that this action is regular. Since every fiber contains $4$ flags, $\bar{G}$ has exactly $4$ orbits on flags. Moreover, since $\bar{G} \leq \Aut(\M^\omega)$, the number $m$ of $\Aut(\M^\omega)$-orbits is either $1$, $2$ or $4$ (a divisor of $4$) and each $\Aut(\M^\omega)$-orbit has a representative in every fibre.

By \Cref{lem:noregular}, $\M^\omega$ is not regular and thus $m \neq 1$. Furthermore, $(u,i) \mapsto (u,i+2)$ is an automorphism of $\M^\omega$. It follows that the four flags contained in any given fiber cannot all belong to different orbits and thus $m \neq 4$. Then, $m$ has no choice but to be $2$.
\end{proof}

\begin{lemma}
\label{lem:stgcover}
    Let $(\M,\omega)$ be a proper pair with $\omega \colon \M \to \ZZ_4$, and let $\C$ be an $\aut(\M)$-invariant $\ell$-colouring of $\M$. If $\M^\omega$ has symmetry type $2^n_I$ then   $(\extm)^{\omega_\C}$ is a maniplex with symmetry type $2^{n+1}_{I\cup\{n\}}$. 
\end{lemma}

\begin{proof}
By \Cref{lem:niceext}, $(\extm,\omega_\C)$ is a proper pair and by \Cref{lem:nicecover} $(\extm)^{\omega_\C}$ is a $2$-orbit maniplex. We will first show that the symmetry-type graphs of $\M^\omega$ and $(\extm)^{\omega_\C}$ coincide in all edge and semi-edges of colour $i < n$. For this, it suffices that we show that there are two $i$-adjacent flags of $\M^\omega$ in the same orbit if and only if there are two $i$-adjacent flags of $(\extm)^{\omega_\C}$ in the same orbit.

The flags of $\M^\omega$ are pairs $(u,a)$ where $u \in \F$ and $a \in \ZZ_4$. Meanwhile, the flags of $(\extm)^{\omega_\C}$ are of the form $((u,x) , a)$, where $u \in \F$, $x \in \ZZ^\ell$ and $a \in \ZZ_4$, but we may slightly abuse notation and write them as triple $(u,x,a)$.

For each $x \in \ZZ_\ell$, let $F_x =  \{ (u,x,a) \mid u\in\F, a\in \ZZ_4\}$. That is, each $F_x$ is (the set of flags of) a facet of $\extmc$. By \Cref{lem:facetasbonitas}, each $F_x$ is isomorphic to $\M^\omega$ via the function 
\[\varphi_x \colon (u,x,a) \mapsto (u,\sigma(x)a). \]

Now, let $i<n$ and suppose there are two $i$-adjacent flags of $\M^\omega$ that lie in the same orbit. Note that this implies, since $\M^\omega$ is a $2$-orbit maniplex, that for any pair of $i$-adjacent flags of $\M^\omega$, there exists an automorphism $\beta  \in \Aut(\M^\omega)$ that interchanges them.

Let $(u,x,a)$ be a flag of $\extmc$. We will show there is an automorphism of $\extmc$ interchanging it with its $i$-neighbour. Note that since $i < n$, the $i$-neighbour of $(u,x,a)$ is a flag of the form $(u^i,x,b)$ for some $b \in \ZZ_4$. Moreover, since $\varphi_x$ is a maniplex isomorphism, the images of $(u,x,a)$ and $(u^i,x,b)$ under $\varphi_x$ are $i$-adjacent flags of $\M^\omega$. Thus there exists $\beta \in \Aut(\M^\omega)$ swapping $(u,x,a)\varphi_x$ and $(u^i,x,b)\varphi_x$. Then for $x \in \ZZ^\ell$, $\Lambda_x \colon = \varphi_x^{-1}\circ \beta \circ \varphi_x$ is an edge-colour preserving permutation of the flags of $F_x$. We want to extend $\Lambda_x$ to an automorphism of $\extmc$. Define 

\[ \Lambda = \prod_{x \in \ZZ^\ell} \Lambda_x \]

and observe that $\Lambda$ is a permutation of the flags of $\extmc$ that preserves all edges of colours $j < n$. Informally, we obtain $\Lambda$ by stitching together all the  $\Lambda_x$ with $x \in \ZZ_\ell$. To see that $\Lambda$ is an automorphism of $\extmc$ it remains to show that it maps $n$-adjacent flags to $n$-adjacent flags. Consider a flag $(u,x,a)$
and note that its $n$-neighbour is of the form $(u,x^j,b)$ for some $x^j \in \ZZ^\ell$ and $b \in \ZZ_4$ (recall that $x^j$ is the element of $\ZZ^\ell$ that differs from $x$ only in its $j$-th entry). Since $\omega_\C$ assigns a weight of $0$ to every $n$-edge of $\extm$, we know that $b = -a$. We also know that $\sigma(x^j) = -\sigma(x)$. Observe that $\varphi_x(u,x,a) = (u,\sigma(x)a) = (u,-\sigma(x^j)(-a)) = \varphi_{x^j}(u,x^j,b)$. Now, $\beta$ maps $(u,\sigma(x)a)$ to some flag $(v,c)$. We have

\[\Lambda(u,x,a) = \Lambda_x(u,x,a) = \varphi_x^{-1} (v,c) = (v,x,\sigma(x)c) \]
and 
\[\Lambda(u,x^j,b) = \Lambda_{x^j}(u,x^j,b) = \varphi_{x^j}^{-1} (v,c) = (v,x^j,\sigma(x^j)c)= (v,x^j,-\sigma(x)c).\]

Since the flags $(v,x,\sigma(x)c)$ and $(v,x^j,-\sigma(x)c)$ are $n$-adjacent in $\extmc$, we see that $\Lambda$ maps $n$-adjacent flags to $n$-adjacent flags. Thus, $\Lambda$ is an automorphism of $\extmc$. We have shown that if for some $i <n$ there are two $i$-adjacent flags of $\M^\omega$ in the same orbit, then in $\extmc$ any two $i$-adjacent flags are in the same orbit.

For the converse, suppose there are for some $i<n$, any two $i$-adjacent flags of $\extmc$ are in the same orbit. Consider two $i$-adjacent flags of $\M^\omega$, $(u,a)$ and $(u^i,b)$, and let $x \in \ZZ^\ell$. Then $(u,a)\varphi_x$ and $(u,b)\varphi_x$ are $i$-adjacent in $\extmc$ and thus there is an automorphism $\gamma$ interchanging them and the composition $\varphi_x \circ \gamma \circ \varphi_x^{-1}$ is an automorphism of $\M^\omega$ mapping $(u,a)$ to its $i$-neighbour $(u^i,b)$. Since $\M^\omega$ has exactly two orbits, it follows that any two $i$-adjacent flags are in the same orbit.

This shows that the symmetry type graphs of $\M^\omega$ and $\extmc$ coincide in all their edges and semi-edges of colours in $\nset$. To complete the proof it remains to show that the symmetry-type graph of $\extmc$ has semi-edges of colour $n$. We will show that there are two $n$-adjacent flags in $\extmc$ belonging to the same orbit.

For this, consider a pair of $n$-adjacent flags in $\extm$, say, $(u,x)$ and $(u,x^j)$ for some $j \in \{1,\ldots,\ell\}$, and recall that the automorphism $\tau_j \in \Aut(\extm)$ (see \Cref{eq:taus}) interchanges $(u,x)$ and $(u,x^j)$. Since $\ZZ_4$ is abelian $\alpha \colon \ZZ_4 \to \ZZ_4$ given by $x \mapsto -x$ is a group automorphism, and by \Cref{cor:autolift}, the mapping $\tilde{\tau}_j \colon (u,x,i) \mapsto (u,x^j,\alpha(i))$ is an automorphism of $\extmc$. Since by definition $\omega_\C$ gives weight $0$ to the $n$-edges of $\extm$, we see that in $\extmc$, the $n$-neighbour of $(u,x,i)$ is $(u,x^j,-i) = (u,x^j,\alpha(i)) = (u,x,i)\tilde{\tau}_j$. That is, $\tilde{\tau}_j$ swaps two $n$-adjacent flags. It follows that the symmetry-type graph of $\extmc$ has semi-edges of colour $n$ and is thus equal to $2^{n+1}_{I\cup{n}}$.
\end{proof}

We now have all the tools to prove the main theorem of this paper.

\section{Construction}
\label{sec:proof}

In this section we will present the construction that will constitute the proof of Theorem \ref{theo:main}. 


For a map $\M$ let $\vartheta \colon \E(\M) \to \ZZ_4$ be given by

\begin{align}
\label{eq: vartheta}
\vartheta(e)=
\begin{cases}
0 & \text{ if $e$ is a $1$-edge}   \\
1 & \text{ otherwise.}
\end{cases}
\end{align}

We will show in \Cref{prop: goodweight} that if $\M$ is non-orientable and regular of type $\{p,q\}$, with one of $p$ or $q$ odd, then $(\M,\vartheta)$ is a proper pair. We will need the following auxiliary lemma and its corollary.

\begin{lemma}
\label{lem:oddcycleevenedges}
Let $\M$ be a non-orientable regular map of type $\{p,q\}$. If one of $p$ or $q$ is odd, then $\M$ contains a closed walk of odd length with an odd number of $1$-edges.
\end{lemma}

\begin{proof}
We can assume $p$ is odd (the case where $q$ is odd follows from a dual argument). Let $W = (u_0,u_1,\ldots,u_{k-1},u_0)$ be a closed walk of odd length $k$. Clearly, $W$ must trace at least one edge of each colour, since it has odd length. Suppose that the number $r>0$ of $1$-edges traced by $W$ is even. Now, consider the longest sequence of edges of alternating colours $0$ and $1$ traced by $W$ and assume, without loss of generality, that the first edge of this sequence is $u_0u_1$. That is, we can assume that for some $a>0$ the walk $A=(u_0,\ldots,u_a)$ traces only edges of colours in $\{0,1\}$ and that this is the longest such walk contained in $W$. Let $B = (u_a,u_{a+1},\ldots,u_{k-1},u_0)$ so that $W$ is simply the concatenation $AB$.

Now, the edge $u_0u_1$ belongs to a $2p$-cycle $C = (u_0,\ldots,u_a,v_0,v_1,\ldots,v_{2p-1},u_0)$ of alternating colours $0$ and $1$. Clearly, $C$ traces exactly $p$ edges of colour $1$. Let $D = (u_a,v_0,v_1,\ldots,v_d,u_0)$ so that $C = AD$.

Let $a$, $b$, and $d$ be the lengths of $A$, $B$ and $D$, respectively. Since $W=AB$ and $W$ has odd length, we see that $a$ and $b$ have different parity. Moreover, $a$ and $d$ have the same parity since $C$ has even length. Then, $BD^{-1}$ is a closed walk of odd length. let $\alpha$, $\beta$ and $\delta$ be the number of $1$-edges traced by $A$, $B$ and $D$. By our assumption, $\alpha+\beta$ is even so $\alpha$ and $\beta$ have the same parity. However, by hypothesis $\alpha+\delta$ is odd (it is equal to $p$), and thus $\beta+\delta$ must be odd. But $\beta+\delta$ is precisely the number of $1$-edges on $BD^{-1}$. This concludes the proof.
\end{proof}

\begin{corollary}
Let $\M$ be a non-orientable regular map of type $\{p,q\}$ where one of $p$ or $q$ is odd. With the weight function $\vartheta$, $\M$ contains a closed walk of odd length and even weight.
\end{corollary}

\begin{proof}
    By \Cref{lem:oddcycleevenedges}, $\M$ has an closed walk of odd length $W$ tracing an even number of $i$-edges with $i \in \{0,2\}$. Each of these edges have weight $1$ and all other edges have weight $0$. Then, the weight of $W$ is a sum of an even amount of terms, each of which is equal to either $1$ or $-1$. It follows that the weight of $W$ is even.
\end{proof}





\begin{proposition}
\label{prop: goodweight}
Let $\M$ be a non-orientable regular map of type $\{p,q\}$ with $p$ or $q$ odd. Let $\vartheta$ be the weight function defined in \Cref{eq: vartheta}. Then $(\M,\vartheta)$ is a proper pair and the cross-cover $\M^\vartheta$ has symmetry-type $2^3_{\{1\}}$.
\end{proposition}

\begin{proof}
We will first show that $(\M,\vartheta)$ is a proper pair. That is, that it satisfies properties (1)--(4) of \Cref{def:properpair}. We get (1) by hypothesis. Clearly, $\vartheta$ is $\Aut(\M)$-consistent since all $i$-edges have the same weight for any given $i \in \{0,1,2\}$, and thus (4) is satisfied. Moreover, by Lemma \ref{lem:oddcycleevenedges}, $\M$ admits a closed walk of odd length and even weight, and thus condition (2) is also satisfied.

We still need to show that condition (3) is satisfied; that is, that $\M^\vartheta$ is a non-orientable maniplex. To see that $\M^\vartheta$ is connected, assume without loss of generality that $p$ is odd. Then, every connected component of $\M_{\{0,1\}}$ is a cycle of alternating colours $0$ and $1$ of length $2p$ and odd weight $p$. It follows that $\M^\vartheta$ is connected by Lemma \ref{lem:evenlift}. Furthermore, $\M^\vartheta$ has the string property by Lemma \ref{lem:extstring}, since all four edges of a cycle $C$ of alternating colour $0$ and $2$ in $\M$ have weight $1$, and thus $\vartheta(C) = 1 -1 +1-1=0$. Therefore, $\M^\vartheta$ is a maniplex. Furthermore, by Lemma \ref{lem:oddcycleevenedges}, $\M$ admits a closed walk of odd length and even weight, and by \Cref{lem:oddlift}, its lift based at $(u,0)$ is a closed walk of odd length in $\M^\vartheta$, which necessarily contains an odd cycle. It follows that $\M^\vartheta$ is non-orientable. Therefore $(\M,\vartheta)$ is a proper pair.

Finally, we need to show $\M^\vartheta$ has symmetry-type $2^3_{\{1\}}$. Let $u$ be a flag of $\M$ and consider the flag $(u,0)$ in the fibre of $u$. Since $\M$ is regular, for $i \in \{0,1,2\}$ there exists $\rho_i \in \Aut(\M)$ mapping $u$ to $u^i$. Observe that all automorphisms of $\M$ are weight preserving, since edges of the same colour have the same weight under $\vartheta$. It follows, that $\rho_i$ has a lift $\tilde{\rho}_i$ mapping $(u,j)$ to $(u^i,j)$ (take $\alpha$ to be the identity in $\Aut(\ZZ_4)$ in \Cref{cor:autolift}). Now, since $1$-edges in $\M$ have trivial weight, we see that $(u,0)$ is $1$-adjacent to $(u^1,0) = (u,0)\tilde{\rho}_1$. This shows that $(u,0)$ is in the same orbit as its $1$-neighbour and since $\M^\vartheta$ has exactly two flag-orbits by \Cref{lem:nicecover}, we see that every flag in $\M^\vartheta$ is in the same orbit as its $1$-neighbour.

Now, let $i \in \{0,2\}$. The $i$-neighbour of $(u,0)$ is $(u^i,1)$, since $i$-edges of $\M$ have weight $1$. Now, $\tilde{\rho_i}$ maps $(u,0)$ to $(u^i,0)$, which belongs to a different orbit than  $(u^i,1)$, by \Cref{lem:nicecover}. Hence $(u,0)$ and its $i$-neighbour belong to different orbits, and since there are only two flag-orbits, the same holds for any flag of $\M$. We have shown that $\M^\vartheta$ has symmetry type $2^3_{\{1\}}$.
\end{proof}

Theorem \ref{theo:main} now follows from Proposition \ref{prop: goodweight} and Lemmas \ref{lem:niceext}, \ref{lem:nicecover} and \ref{lem:stgcover}. Indeed, given a non-orientable regular map $\M$ of type $\{p,q\}$, where $p$ or $q$ is odd, the pair $(\M,\vartheta)$ is a proper pair by \Cref{prop: goodweight}. The cross-cover of $(\M,\vartheta)$ is an unstable $2$-orbit maniplex by \Cref{lem:nicecover}, and it has symmetry-type $2^3_{\{1\}}$ by \Cref{lem:stgcover}, making it fully-transitive.

Furthermore, by \Cref{lem:niceext}, $(\M,\vartheta)$ can be extended by means of the colouring $\hat{\C}$ to a proper pair $(\hat{2}^\M,\vartheta_{\hat{\C}})$ of rank $4$, whose cross-cover  is an unstable $2$-orbit maniplex of symmetry-type $2^3_{\{1\}}$, by \Cref{lem:nicecover} and \Cref{lem:stgcover}. 
In turn, the rank $4$ extension $(\hat{2}^\M,\omega_{\hat{\C}})$ (and each of its subsequent extensions) can be extended to a proper pair in two different ways by using either its total colouring $\hat{\C}$ or its antipodal colouring $\C^\times$. This produces an infinite family of proper pairs, with $2^{n-3}$ elements for each rank $n > 3$. The cross-cover of each such proper pair is an unstable fully-transitive maniplex with two orbits, of symmetry-type is $2^{n}_{\{\overline{0,2}\}}$.

\begin{figure}[H]
    \centering
    \includegraphics[width=0.7\linewidth]{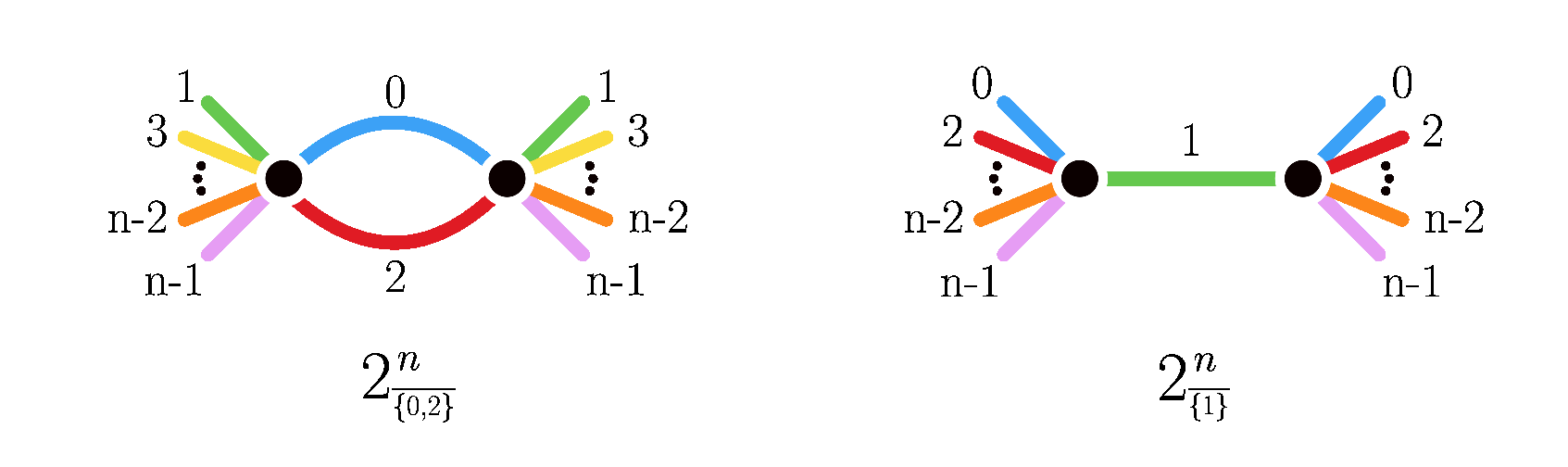}
    \caption{The symmetry-type graphs $2^{n}_{\{\overline{0,2}\}}$ and $2^{n}_{\{\overline{1}\}}$}
    \label{fig:stgs}
\end{figure}

A variation of this construction yields a similar infinite family of $2$-orbit unstable maniplexes, but with symmetry type $2^{n}_{\{\overline{1}\}}$. This can be done by considering the $4$-weight function $\vartheta'$ that gives a weight of $1$ to all $1$-edges, and null weight to all other edges (in a sense, a `reversed' $\vartheta$). An analog of Proposition \ref{prop: goodweight} holds for $\vartheta'$, with the difference that the resulting proper pair $(\M,\vartheta')$ has a cross-cover of symmetry type $2^3_{\{0,2\}}$, instead of $2^3_{\{1\}}$. Thus, a different infinite family of $2$-orbit (although not fully-transitive) unstable maniplexes with symmetry type $2^3_{\overline{\{1\}}}$ can be produced.


%
%
%
%
%
%
%
%
%
%
%

\bibliographystyle{acm} 
\bibliography{inestabilidad}

\end{document}